\documentclass[11pt,twoside,a4paper]{article}

\usepackage{algorithm}
\usepackage{algpseudocode}
\usepackage{graphicx,url}
\usepackage{setspace}
\usepackage{enumerate,lineno,setspace,float}
\usepackage[small,compact]{titlesec}
\usepackage{amsmath,amsfonts,amssymb,amsthm}
\usepackage{geometry}
\usepackage{fancyhdr}
\usepackage{amssymb,amsmath}
\usepackage{latexsym}
\usepackage{graphicx, float}
\usepackage{url}

\newtheorem{Theorem}{Theorem}[section]

\newtheorem{Lemma}[Theorem]{Lemma}
\newtheorem{Example}[Theorem]{Example}
\newtheorem{Observation}[Theorem]{Observation}

\usepackage{color}

\definecolor{Blue}{rgb}{0,0,1}
\definecolor{Red}{rgb}{1,0,0}

\long\def\delete#1{}    

\input amssym.def
\input amssym.tex

\newcommand{\be}{\begin{equation}}
\newcommand{\ee}{\end{equation}}
\newcommand{\bea}{\begin{eqnarray}}
\newcommand{\eea}{\end{eqnarray}}
\newcommand{\bean}{\begin{eqnarray*}}
\newcommand{\eean}{\end{eqnarray*}}

\def\la{\langle}

\def\diag{{\rm diag}}

\def\b{\beta}

\def\l{\lambda}

\def\p{\phi}
\def\r{\rho}
\def\I{\mbox{\rm \textbf{I}}}
\def\R{\mbox{\textbf{R}}}

\def\la{\leftarrow}

\def\({\left(}
\def\){\right)}
\def\[{\left[}
\def\]{\right]}

\begin{document}

\title{A complete characterization of spectra of the Randić matrix of level-wise regular trees}
\author{\textbf{Punit Vadher} \\ Gujarat Technogical University, Ahmedabad - 382 424, Gujarat (INDIA) \\ E-mail : punitv99@gmail.com \\ \\
\textbf{Devsi Bantva}\footnote{Corresponding author.} \\ Lukhdhirji Engineering College, Morvi - 363 642, Gujarat (INDIA) \\ E-mail : devsi.bantva@gmail.com}

\pagestyle{myheadings}
\markboth{\centerline{Punit Vadher and Devsi Bantva}}{\centerline{A complete characterization of spectra of the Randić matrix of level-wise regular trees}}

\date{}
\openup 0.8\jot
\maketitle

\begin{abstract}
Let $G$ be a simple finite connected graph with vertex set $V(G) = \{v_1,v_2,\ldots,v_n\}$. Denote the degree of vertex $v_i$ by $d_i$ for all $1 \leq i \leq n$. The Randić matrix of $G$, denoted by $\R(G) = [r_{i,j}]$, is the $n \times n$ matrix whose $(i,j)$-entry $r_{i,j}$ is $r_{i,j} = 1/\sqrt{d_id_j}$ if $v_i$ and $v_j$ are adjacent in $G$ and 0 otherwise. A tree is a connected acyclic graph. A level-wise regular tree is a tree rooted at one vertex $r$ or two (adjacent) vertices $r$ and $r'$ in which all vertices with the minimum distance $i$ from $r$ or $r'$ have the same degree $m_i$ for $0 \leq i \leq h$, where $h$ is the height of $T$. In this paper, we give a complete characterization of the eigenvalues with their multiplicity of the Randić matrix of level-wise regular trees. We prove that the eigenvalues of the Randić matrix of a level-wise regular tree are the eigenvalues of the particular tridiagonal matrices, which are formed using the degree sequence $(m_0,m_1,\ldots,m_{h-1})$ of level-wise regular trees. \\


\end{abstract}

\section{Introduction}
The graphs considered in this paper are finite connected graphs  without loops or multiple edges. Denote $V(G) = \{v_1, v_2,\dots, v_ n\}$ the vertex set of graph $G$ and $d_i$ the degree of vertex $v_i$ for $i = 1,2,\ldots,n$ (the degree of a vertex $v \in V(G)$ is the number of adjacent vertices to $v$). In 1975, Randić \cite{Randi'c 1975} invented a molecular structure descriptor defined as 
$$
\R\I(G) =  \sum_{v_i \sim v_j}^{} \frac{1}{\sqrt {d_id_j}}, 
$$
where the summation runs over all pairs of adjacent vertices of the underlying (molecular) graph. This graph invariant is nowadays known as Randić index named after Randić; for details refer \cite{Gutman2008,Li2006,Randi'c 2008,Rodriguez2005A,Rodriguez2005B}. The readers are advised to refer \cite{Li2008} for a detail survey on the Randić index. Inspired by the Randić index, a symmetric square matrix $\R = \R(G)$ of order $n$ is associated to the graph $G$ known as Randić matrix of $G$, whose $(i,j)$-entry $r_{i,j}$ is defined as
$$
r_{i,j} = \left\{
\begin{array}{ll}
\frac{1}{\sqrt {d_id_j}}, & \mbox{ if $v_i$ and $v_j$ are adjacent vertices in $G$} , \\ [0.2cm]
0 , & \mbox{ otherwise}.
\end{array}
\right.
$$
The Randić matrix $\R(G)$ without this name was first studied in the book \cite{Cvetkovic1980} by Cvetković, Doob and Sachs (refer \cite{Gutman2014}). The matrix $\R(G)$ was also used to study Randić index in 2005 by Rodríguez who called it the weighted adjacency matrix in \cite{Rodriguez2005A} and the degree adjacency matrix in \cite{Rodriguez2005B}. The $\R$-characteristic polynomial of graph $G$ is the $\R$-characteristic polynomial of the Randić matrix $\R$ of graph $G$ which is $\R(\l) = \det(\l\I_n-\R)$, where $\I_n$ is the identity matrix of size $n$. The Randić matrix is a real symmetric matrix and hence the eigenvalues of $\R$ are real numbers. Denote the eigenvalues of Randić matrix $\R$ by $\r_1 \geq \r_2 \geq \ldots \geq \r_n$ with multiplicity $m_1(\r_1),m_2(\r_2),\ldots,m_n(\r_n)$. The $\R$-spectrum of graph $G$, denoted by $S_{R}(G)$, is 
$$
S_{R}(G) = \bigg(\begin{matrix}
\r_{1} & \r_{2} & \r_{3} & ... & \r_{n} \\
m_{1}(\r_{1}) & m_{2}(\r_{2}) & m_{3}(\r_{3}) & ... & m_{n}(\r_{n})
\end{matrix}\bigg).
$$
The relation between the eigenvalues of Randić matrix of graph $G$ and  Randić index of graph $G$ is as follows (see \cite{Bozkurt2010}): 
$$
\sum_{i=1}^n\r_i^2 = 2 \cdot \R\I(G).
$$
In \cite{Gutman2014}, Gutman et al. gave the following important observations about $\R(G)$ (and it is worth to recall here in connection with our results) while studying the Randić energy of graph $G$:
\begin{itemize}
\item If $G$ has $k$ components, namely $G_1,G_2,\ldots,G_k$ then $S_R(G) = S_R(G_1) \cup S_R(G_2) \cup \ldots \cup S_R(G_k)$.
\item Let $\overline{K}_n$ be the complement of complete graph $K_n$ then $S_R(\overline{K}_n) = \bigg(\begin{matrix}
0 \\
n
\end{matrix}\bigg).$
\item Let $G$ be a graph on $n \geq 1$ vertices and let $\r_1$ be the greatest eigenvalue of its Randić matrix. Then $\r_1=0$ holds if and only if $G \cong \overline{K}_n$. If $G$ has at least one edge, then $\r_1 = 1$ (the last statement is also proved in \cite[Theorem 2.1]{Liu2012}).
\end{itemize}

Moreover, Furtula and Gutman noted in \cite{Furtula2013} that the sum of eigenvalues of Randić matrix of graph $G$ is equal to zero and the relations $$\r_i = \r_{n-i+1},\;\; i=1,2,\ldots,n$$ holds  if and only if the graph $G$ is bipartite.  

The spectra of the Randić matrix is studied by many authors for different purposes, like Randić index, Randić energy, structural information of graph etc. (see \cite{Alikhani2015,Dehmer2015,Gao2021}). In \cite{Alikhani2015}, Alikhani and Ghanbari gave the $\R$-characteristic polynomial (using it one can easily find $\R$-spectrum) for path $P_n$, star $K_{1,n-1}$, cycle $C_n$, complete graph $K_n$, complete bipartite graph $K_{m,n}$ and friendship graphs $F_n$ and $D_4^n$ (refer \cite{Alikhani2015} for definitions of friendship graphs $F_n$ and $D_4^n$). In \cite{Yin2021}, Yin studied the multiple Randić eigenvalues of trees. Andrade et al. studied the Randić spectra of caterpillar graphs (a tree is called a caterpillar if the removal of all its degree-one vertices results in a path) in \cite{Andrade2017}. Rather et al. studied the Randić spectrum of zero divisor graphs of commutative ring $\mathbb{Z}_n$ in \cite{Rather2023}. The adjacency spectra and Laplacian spectra  of level-wise regular trees are studied in \cite{Rojo2002} and \cite{Rojo2005}, respectively. Fernandes et al. discussed the spectra of some graphs like weighted trees in \cite{Fernandes2008}.  

In this paper, we completely characterize all the eigenvalues (with their multiplicity) of the Randić matrix of the level-wise regular trees, a large class of trees. In fact, we prove that the eigenvalues of the Randić matrix of level-wise regular trees are the eigenvalues of the leading principal submatrices of three symmetric tridiagonal matrices of order equal to the height of level-wise regular trees plus one. This paper is organised as follows: We define the necessary terms, notations and also recall some well-known results of spectral theory in section \ref{Prel}. We represent the Randić matrix of level-wise regular trees as a tridiagonal matrix of blocks, which we used to prove our main result. The section \ref{Main} consists of three Lemmas (which are useful to prove our main results) and two main Theorems. We illustrated the whole procedure of main results with the help of two examples.     

\section{Preliminaries}\label{Prel}

In this section, we define necessary terms and notations to prove our main results. We also recall some standard results that are necessary for the present work. 

A tree is a connected acyclic graph. A star on $n$ vertices, denoted by $K_{1,n-1}$, is a tree consisting of $n-1$ leaves and another vertex joined to all leaves by edges. A level-wise regular tree is a tree rooted at one vertex $r$ or two (adjacent) vertices $r$ and $r'$ in which all vertices with the minimum distance, $i$ from $r$ or $r'$ have the same degree $m_i$ for $0 \leq i \leq h$, where $h$ is the height of $T$. Denote the level-wise regular tree by $T^1 = T^1_{m_0,m_1,\ldots,m_{h-1}}$ (with one root) and $T^2 = T^2_{m_0,m_1,\ldots,m_{h-1}}$ (with two root). Observe that the diameter of $T^1_{m_0,m_1,\ldots,m_{h-1}}$ is $2h$ and $T^2_{m_0,m_1,\ldots,m_{h-1}}=2h+1$. The order of level-wise regular tree $T^z=T^z_{m_0,m_1,\ldots,m_{h-1}}\;(z=1,2)$ is  
\begin{equation}\label{eq:order}
|V(T^z)| = \left\{
\begin{array}{ll}
1+m_0+m_0\sum_{t=1}^{h-1}\(\prod_{k=1}^{t}(m_k-1)\), & \mbox{if $z=1$}, \\ [0.2cm]
2+2\sum_{t=0}^{h-1}\(\prod_{k=0}^{t}(m_k-1)\) , & \mbox{if $z=2$}.
\end{array}
\right.
\end{equation}
In this work by a tree $T$ we mean a level-wise regular tree $T^z_{m_0,m_1,\ldots,m_{h-1}}$, where $z=1,2$. Moreover, we assume $m_0 \geq 2$ for $T^1 = T^1_{m_0,m_1,\ldots,m_{h-1}}$ and $m_0 \geq3$ for $T^2 = T^2_{m_0,m_1,\ldots,m_{h-1}}$. Denote $C(T^1) = \{r\}$ and $C(T^2) = \{r,r'\}$. In both cases, we put the $C(T^z)$, where $z=1,2$ at level $0$ then $T^z$ has $h$ levels. Thus the vertices in the level $h$ have degree $m_h = 1$. Define $h+1$ layers $L_i, 0 \leq i \leq h$ of level-wise regular tree $T^z = T^z_{m_0,m_1,\ldots,m_{h-1}} (z=1,2)$ as 
$$L_{i} = \{v \in V(T) : \min\{d(v,u) : u \in C(T^z)\} = i\}.$$ 
Observe that $|L_0|=1, |L_1| = m_0$ for $T^1$ and $|L_0| = 2, |L_1| = 2(m_0-1)$ for $T^2$ and for $2 \leq i \leq h$, 
\begin{equation*}
|L_i| = \left\{
\begin{array}{ll}
m_0\prod_{k=1}^{i-1}(m_k-1), & \mbox{if $|L_0|=1$}, \\ [0.2cm]
2\prod_{k=0}^{i-1}(m_k-1), & \mbox{if $|L_0|=2$}.
\end{array}
\right.
\end{equation*}
Moreover, in a level-wise regular tree $T^z_{m_0,m_1,\ldots,m_{h-1}}\;(z=1,2)$, it is clear that for any $1 \leq i \leq h-1$, $|L_{i+1}| = (m_i-1)|L_i|$ and, $|L_i| = |L_{i+1}|$ if and only if $m_i=2$. Observe that for all $j=0,1,\ldots,h-1$, $L_j$ divides $L_{j+1}$.

We apply Algorithm \ref{Algo} on vertices of $T^z = T^z_{m_0,m_1,\ldots,m_{h-1}}\;(z=1,2)$ of order $n$ to give an ordering $v_1,v_2,\ldots,v_n$ of $V(T^z)$ to write Randić matrix of $T^z$.  
\begin{algorithm}
\caption{A vertex ordering $v_1,v_2,\ldots,v_n$ of $V(T^z),\;z=1,2$.}\label{Algo}
\hspace*{\algorithmicindent} \textbf{Input:} A vertex set $V(T^z)$ of a tree $T^z$, where $z=1,2$.
\begin{algorithmic}[1]
\State $v_1 \la u$ for any $u \in L_h$.
\State{Set $k_0=0$ and $k_i = \sum_{t=1}^{i}|L_{h-i+1}|$ for $1 \leq i \leq h-1$.}
\For{$1 \leq i \leq h-1$}
\For{$1+k_{i-1} \leq j \leq k_i-1$}
\State{$v_{j+1} \la u$ for a $u \in L_{h-i+1}$ such that $d(v_j,v_{j+1}) \leq d(v_j,w)$ for any $w \in L_{h-i+1} \setminus \{v_1,v_2,\ldots,v_j,u\}$ and $v_{k_i+1}$ is a parent of $v_{k_{i-1}+1}$.}
\EndFor
\EndFor
\State $v_n \la r$ if $L_0 = \{r\}$ and $v_{n-1} \la r, v_n \la r'$ such that $d(v_1,v_{n-1}) = h$ and $d(v_1,v_n) = h+1$ if $L_0 = \{r,r'\}$.
\State \textbf{return} $\vec{V}:=\{v_1,v_2,\ldots,v_{n}\}$.
\end{algorithmic}
\hspace*{\algorithmicindent} \textbf{Output:} A linear order $v_1,v_2,\ldots,v_n$ of $V(T^z),\;z=1,2$.
\end{algorithm}

The readers can view the vertex ordering of tree $T^z_{m_0,m_1,\ldots,m_{h-1}}\;(z=1,2)$ describe in Algorithm \ref{Algo} in Examples \ref{EX:1} and \ref{EX:2}, respectively. We fix the ordering $\{v_1,v_2,\ldots,v_n\}$ of vertices of level-wise regular tree $T^z_{m_0,m_1,\ldots,m_{h-1}}\;(z=1,2)$ obtained by applying Algorithm \ref{Algo} for all our subsequent discussion. 

For $j = 1,2,\ldots,h$, define $C_j$ is the column matrix of order $\frac{|L_{h-j+1}|}{|L_{h-j}|} \times 1$ with each entry is equal to $\frac{1}{\sqrt{m_{h-j+1}m_{h-j}}}$. That is, 
\begin{equation*}\label{mat:Cj}
C_j = 
\begin{bmatrix}
\frac{1}{\sqrt{m_{h-j}m_{h-j+1}}} & \\
\frac{1}{\sqrt{m_{h-j}m_{h-j+1}}} & \\
\vdots & \\
\frac{1}{\sqrt{m_{h-j}m_{h-j+1}}} &  
\end{bmatrix}_{\frac{|L_{h-j+1}|}{|L_{h-j}|} \times 1}.   
\end{equation*}

For $j=1,2,\ldots,h$, let $B_j$ is the block diagonal matrix of order $|L_{h-j+1}| \times |L_{h-j}|$ defined by 
\begin{equation*}\label{mat:Bj}
B_j = 
\begin{bmatrix}
C_j & 0 & \cdots & 0\\
0 & C_j  & &\vdots\\
\vdots & & \ddots & 0\\
0 & \cdots & 0 & C_j 
\end{bmatrix}_{|L_{h-j+1}| \times |L_{h-j}|}   
\end{equation*}


\begin{Observation}\label{Obs1} Let $T^z_{m_0,m_1,\ldots,m_{h-1}}(z=1,2)$ be the level-wise regular tree. Then the following hold for all $1 \leq j \leq h$.
\begin{enumerate}[\rm (a)]
\item $|L_0| \leq |L_1| \leq \ldots \leq |L_h|$ and for $1 \leq j \leq h$, $\frac{|L_{h-j+1}|}{|L_{h-j}|}$ is a positive integer,
\item $C_j^TC_j = \frac{|L_{h-j+1}|}{|L_{h-j}|}\frac{1}{m_{h-j}m_{h-j+1}}$,
\item $B_j^{T}B_j = \begin{bmatrix}
C_j^TC_j & 0 & \cdots & 0\\
0 & C_j^TC_j  & &\vdots\\
\vdots & & \ddots & 0\\
0 & \cdots & 0 & C_j^TC_j
\end{bmatrix} = \frac{|L_{h-j+1}|}{|L_{h-j}|}\frac{1}{m_{h-j}m_{h-j+1}}\I_{|L_{h-j}|}$, 
\item $\det(B_j^{T}B_j) = \(\frac{|L_{h-j+1}|}{|L_{h-j}|}\frac{1}{m_{h-j}m_{h-j+1}}\)^{|L_{h-j}|}.$
\end{enumerate}
\end{Observation}

We illustrate the above defined terms and notations with the help of the following example of $T^1_{4,4,3}$ and $T^2_{4,3,4}$ for more clarity to readers.

\begin{Example}\label{EX:1} Let $T^1 = T^1_{4,4,3}$ be the level-wise regular tree of height $3$ as shown in the following Figure \ref{Fig1}.
\end{Example}
\begin{figure}[ht!]
\centering
\includegraphics[width=5in]{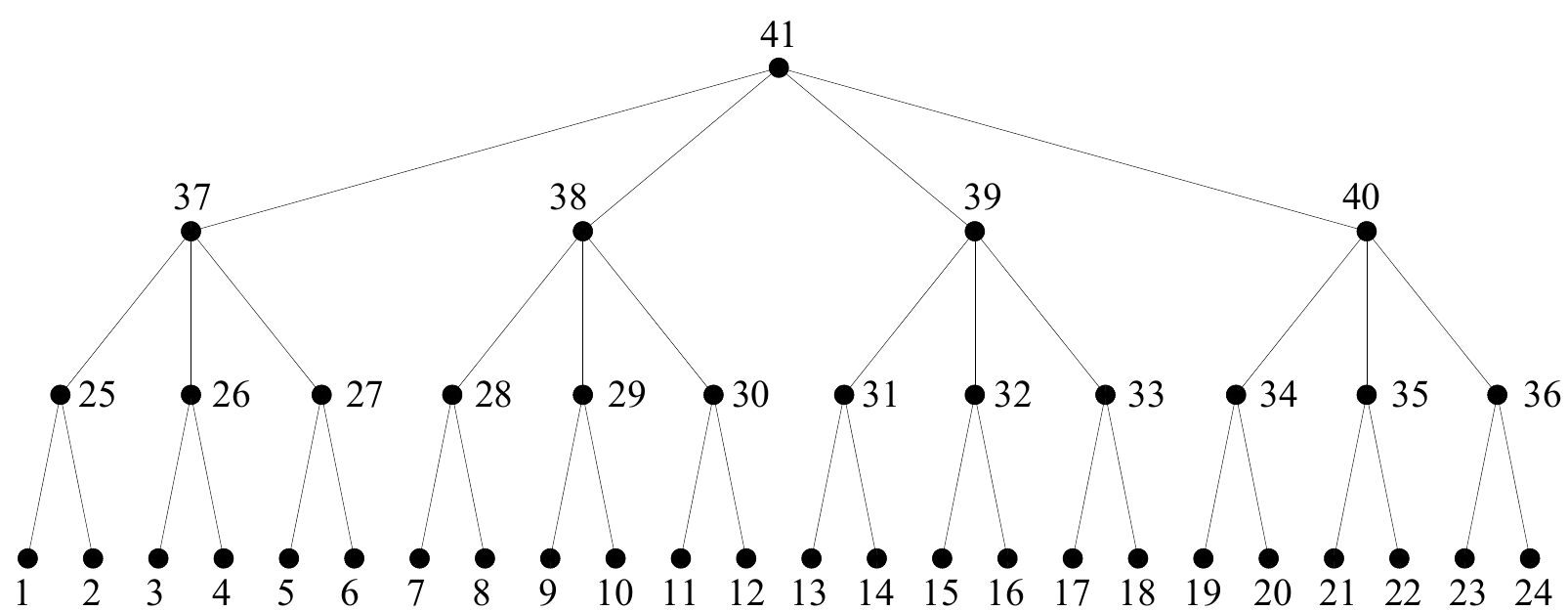}
\caption{Tree $T^1_{4,4,3}$ and its vertex ordering.}\label{Fig1}
\end{figure} 
Note that a tree $T^1_{4,4,3}$ has $4$ levels, $|L_0|=1, |L_1| = 4, |L_2|=12, |L_3| = 24$ and the vertex degrees are $m_0 = 4, m_1 = 4, m_2 = 3, m_3 = 1$. Then we have $\frac{|L_3|}{|L_2|} = 2, \frac{|L_2|}{|L_1|} = 3$ and $\frac{|L_1|}{|L_0|} = 4$. Moreover, for $j = 1,2,3$, the matrices $C_j$ are
$$C_1 = \begin{bmatrix}
\frac{1}{\sqrt3}\\[0.2cm]
\frac{1}{\sqrt3}
\end{bmatrix}_{2 \times 1}, C_2 = \begin{bmatrix}
\frac{1}{2\sqrt{3}}\\[0.2cm]
\frac{1}{2\sqrt{3}}\\[0.2cm]
\frac{1}{2\sqrt{3}}
\end{bmatrix}_{3 \times 1} \mbox{ and }C_3 = \begin{bmatrix}
\frac{1}{4}\\[0.1cm]
\frac{1}{4}\\[0.1cm]
\frac{1}{4}\\[0.1cm]
\frac{1}{4}
\end{bmatrix}_{4 \times 1}.$$
Hence, 
$$B_1 = \diag [C_1,C_1,C_1,C_1,C_1,C_1,C_1,C_1,C_1,C_1,C_1,C_1]$$ 
$$B_2 = \diag [C_2,C_2,C_2,C_2]$$ 
and 
$$B_3 = \diag [C_3].$$
The Randić matrix of the tree $T^1_{4,4,3}$ is 
$$\R(T^1) = \begin{bmatrix}
0 & B_1 & 0 & 0\\
B_1^{T} & 0 & B_2 & 0\\
0 & B_2^{T} & 0 & B_3\\
0 & 0 & B_3^{T} & 0
\end{bmatrix}.$$

\begin{Example}\label{EX:2} Let $T^2 = T^2_{4,3,4}$ be the level wise regular tree of height $3$ as shown in the following Figure \ref{Fig2}.
\end{Example}
\begin{figure}[ht!]
\centering
\includegraphics[width=6in]{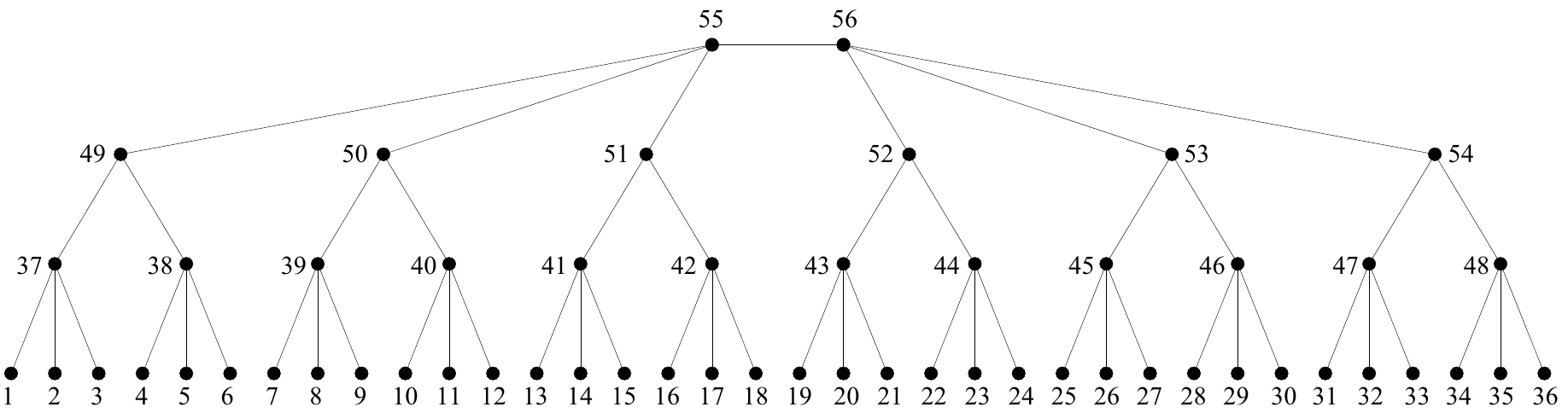}
\caption{Tree $T^2_{4,3,4}$ and its vertex ordering.}\label{Fig2}
\end{figure}
Note that a tree $T^2_{4,3,4}$ has $4$ levels, $|L_0| = 2$, $|L_1| = 6$, $|L_2| = 12$, $|L_3| = 36$ and the vertex degrees are $m_0=4$, $m_1=3$, $m_2=4$, $m_3=1$. Then we have $\frac{|L_3|}{|L_2|} = 3$, $\frac{|L_2|}{|L_1|}=2$ and $\frac{|L_1|}{|L_0|}=3$. Moreover, for $j=1,2,3$, the matrices $C_j$ are 
$$C_1 = \begin{bmatrix}
\frac{1}{2}\\[0.2cm]
\frac{1}{2}\\[0.2cm]
\frac{1}{2}
\end{bmatrix}_{3 \times 1}, C_2 = \begin{bmatrix}
\frac{1}{2\sqrt{3}}\\[0.2cm]
\frac{1}{2\sqrt{3}}
\end{bmatrix}_{2 \times 1} \mbox{ and }C_3 = \begin{bmatrix}
\frac{1}{2\sqrt{3}}\\[0.1cm]
\frac{1}{2\sqrt{3}}\\[0.1cm]
\frac{1}{2\sqrt{3}}
\end{bmatrix}_{3 \times 1}.$$
Hence, 
$$B_1 = \diag [C_1,C_1,C_1,C_1,C_1,C_1,C_1,C_1,C_1,C_1,C_1,C_1]$$ 
$$B_2 = \diag [C_2,C_2,C_2,C_2,C_2,C_2]$$ 
$$B_3 = \diag [C_3,C_3]$$
and let  
$$\Omega = \begin{bmatrix}
    0 & \frac{1}{4} \\
    \frac{1}{4} & 0
\end{bmatrix}.$$
The Randić matrix of the tree $T^2_{4,3,4}$ is 
$$\R(T^2) = \begin{bmatrix}
0 & B_1 & 0 & 0 \\
B_1^{T} & 0 & B_2 & 0 \\
0 & B_2^{T} & 0 & B_3 \\
0 & 0 & B_3^{T} & \Omega
\end{bmatrix}.$$

In general, the Randić matrix of level-wise regular tree $T^z = T^z_{m_0,m_1,\ldots,m_{h-1}}\;(z=1,2)$ of order $n$ is given by
$$\R(T^z) = \begin{bmatrix}
0 & B_1 & 0 & \cdots & \cdots & 0\\
B_1^{T} & 0 & B_2 & \ddots & & \vdots\\
0 & B_2^{T} & \ddots & \ddots & \ddots & \vdots \\
\vdots & \ddots & \ddots & \ddots & B_{h-1} & 0\\
\vdots & & \ddots & B_{h-1}^{T} & 0 & B_{h}\\
0 & \cdots & \cdots & 0 & B_{h}^{T} & \Omega
\end{bmatrix},$$ 
where 
$$
\Omega = \left\{
\begin{array}{ll}
0, & \mbox{ if $z=1$} , \\ [0.2cm]
\begin{bmatrix}
0 & \frac{1}{m_0} \\
\frac{1}{m_0} & 0
\end{bmatrix}, & \mbox{ if $z=2$}.
\end{array}
\right.
$$

We recall some well known results (refer \cite{Cvetkovic1980} and \cite{Trefethen1997}) which are useful to prove our main results.

\begin{Theorem}[Eigenvalue Interlacing Theorem](refer \cite{Cvetkovic1980}) \label{thm:interl} Suppose $A \in \mathbb{R}^{n \times n}$ is symmetric. Let $B \in \mathbb{R}^{m \times m}$ with $m < n$ be a principal submatrix (obtained by deleting both $i$-th row and $i$-th column for some value of $i$). Suppose $A$ has eigenvalues $\l_1 \leq \ldots \leq \l_n$ and $B$ has eigenvalues $\l'_1 \leq \ldots \leq \l'_m$. Then 
$$\l_k \leq \l'_k \leq \l_{k+n-m} \mbox{ for } k=1,2,\ldots,m.$$
And if $m=n-1$ then 
$$\l_1 \leq \l'_1 \leq \l_2 \leq \l'_2 \leq \ldots \leq \l'_{n-1} \leq \l_n.$$
\end{Theorem}

\begin{Lemma}\cite{Trefethen1997}\label{Prop1} Let 
$$M_j = \begin{bmatrix}
p_1 & q_1 & 0 & \cdots & \cdots & 0\\
q_1 & p_2 & q_2 & \ddots & & \vdots\\
0 & q_2 & \ddots & \ddots & \ddots & \vdots \\
\vdots & \ddots & \ddots & \ddots & \ddots & 0\\
\vdots & & \ddots & \ddots & p_{j-1} & q_{j-1}\\
0 & \cdots & \cdots & 0 & q_{j-1} & p_j
\end{bmatrix}$$
and $M_j(\l) = \det(\l\I_j-M_j)$. Then $M_j(\l)$ satisfies the following three-term recursion formula
\begin{equation*}
M_j(\lambda)=(\lambda-p_j)M_{j-1}(\lambda)-q^2_{j-1}M_{j-2}(\lambda)
\end{equation*} with $M_0(\lambda)=1$ and $M_1(\lambda)=\lambda-a_1$.
\end{Lemma}

\begin{Lemma}\label{lem2} Let $T^z=T^z_{m_0,m_1,\ldots,m_{h-1}}\;(z=1,2)$ be the level-wise regular tree, $L_i$'s are defined as earlier and 
$$D = \begin{bmatrix}
-\I_{|L_h|}&0&0&\cdots&\cdots&0\\
0&\I_{|L_{h-1}|}&0&\ddots&\ddots&\vdots\\
0&0&-\I_{|L_{h-2}|}&\ddots&\ddots&\vdots\\
\vdots&\ddots&\ddots&\ddots&\ddots0\\
\vdots&\ddots&\ddots&\ddots&(-1)^{h}\I_{|L_{1}|}&0\\
0&\cdots&\cdots&0&0&(-1)^{h+1}\I_{|L_0|}
\end{bmatrix} 
$$ 
then $$D(\l\I+\R(T^z))D^{-1}=\l \I-\R(T^z).$$
In particular, $$\det(\l\I-\R(T^z)) = \det(\l\I+\R(T^z)).$$ 
\end{Lemma}
\begin{proof} It is clear that $D = D^{-1}$. Computing the multiplication block wise, we have $D(\l\I+\R(T^z))D^{-1}$
$$= \begin{bmatrix}
-\I_{|L_h|}\cdot \l\I_{|L_h|} \cdot -\I_{|L_h|}&-\I_{|L_h|}\cdot B_1 \cdot \I_{|L_{h-1}}|&\cdots&\cdots&0\\
\I_{|L_{h-1}|} \cdot B_1^{T} \cdot -\I_{|L_h|}&\I_{|L_{h-1}|}\cdot \l\I_{|L_{h-1}|}\cdot \I_{|L_{h-1}|}&\ddots&\ddots&\vdots\\
\vdots&\ddots&\ddots&\ddots & \vdots\\
\vdots&\ddots&\ddots&\ddots& \vdots\\
0&\cdots&\cdots& \cdots &(-1)^{h+1}\I_{|L_0|} (\l\I_{|L_0|}-\Omega)(-1)^{h+1}\I_{|L_0|}
\end{bmatrix} 
$$ 
$$\hspace{-5.3cm} = \begin{bmatrix}
\l\I_{|L_h|} & -B_1 & 0 & \cdots & \cdots & 0\\
-B_1^{T} & \l\I_{|L_{h-1}|} & -B_2 & \ddots & & \vdots\\
0 & -B_2^{T} & \ddots & \ddots & \ddots & \vdots \\
\vdots & \ddots & \ddots & \ddots & -B_{h-1} & 0\\
\vdots & & \ddots & -B_{h-1}^{T} & \l\I_{|L_1|} & -B_{h}\\
0 & \cdots & \cdots & 0 & -B_{h}^{T} & \l\I_{|L_0|}-\Omega
\end{bmatrix}.$$
= $\l\I-\R(T^z)$.
\end{proof}
 
\section{The spectrum of the Randić matrix of level-wise regular tree}\label{Main}

We continue to use terms and notations defined in the previous sections. We prove the following main result in this section. 
\begin{Theorem}\label{thm:main} Let $T^z=T^z_{m_0,m_1,\ldots,m_{h-1}}\;(z=1,2)$ be the level-wise regular tree, $L^i$'s are defined as earlier and $\psi = \{j : j \in \{1,2,\ldots,h\} \mbox{ and } |L_{h-j+1}| > |L_{h-j}|\}$. For $j=1,2,\ldots,h$, let $P_j$ be the $j \times j$ principal sub matrix of the $(h+1) \times (h+1)$ symmetric tridiagonal matrix 
$$P_{h+1} = \begin{bmatrix}
0 & \sqrt{\left(\frac{m_{h-1}-1}{m_hm_{h-1}}\right)} & 0 & \cdots & \cdots & 0\\
\sqrt{\left(\frac{m_{h-1}-1}{m_hm_{h-1}}\right)} & 0 & \sqrt{\left(\frac{m_{h-2}-1}{m_{h-1}m_{h-2}}\right)} & \ddots & & \vdots\\
0 & \sqrt{\left(\frac{m_{h-2}-1}{m_{h-1}m_{h-2}}\right)} & \ddots & \ddots & \ddots & \vdots \\
\vdots & \ddots & \ddots & \ddots & \sqrt{\left(\frac{m_1-1}{m_2m_1}\right)} & 0\\
\vdots & & \ddots & \sqrt{\left(\frac{m_1-1}{m_2m_1}\right)} & 0 & \sqrt{\frac{1}{m_1}}\\
0 & \cdots & \cdots & 0 & \sqrt{\frac{1}{m_{1}}} & 0
\end{bmatrix}.$$
and for $z=1,2$,
$$P_{h+1}^z = \begin{bmatrix}
0 & \sqrt{\left(\frac{m_{h-1}-1}{m_hm_{h-1}}\right)} & 0 & \cdots & \cdots & 0\\
\sqrt{\left(\frac{m_{h-1}-1}{m_hm_{h-1}}\right)} & 0 & \sqrt{\left(\frac{m_{h-2}-1}{m_{h-1}m_{h-2}}\right)} & \ddots & & \vdots\\
0 & \sqrt{\left(\frac{m_{h-2}-1}{m_{h-1}m_{h-2}}\right)} & \ddots & \ddots & \ddots & \vdots \\
\vdots & \ddots & \ddots & \ddots & \sqrt{\left(\frac{m_1-1}{m_2m_1}\right)} & 0\\
\vdots & & \ddots & \sqrt{\left(\frac{m_1-1}{m_2m_1}\right)} & 0 & \sqrt{\left(\frac{m_0-1)}{m_1m_0}\right)}\\
0 & \cdots & \cdots & 0 & \sqrt{\left(\frac{m_0-1}{m_{1}m_0}\right)} & \frac{(-1)^z}{m_0}
\end{bmatrix}.$$
Then 
\begin{enumerate}[\rm (a)]
\item $\sigma(\R(T^1))=\left(\displaystyle\bigcup_{j\in\psi} \sigma(P_j) \right) \cup \sigma(P_{h+1})$ and $\sigma(\R(T^2))=\left(\displaystyle\bigcup_{j\in\psi} \sigma(P_j) \right) \cup \sigma(P^1_{h+1}) \cup \sigma(P^2_{h+1})$.
\item The multiplicity of each eigenvalue of the matrix $P_j$ as an eigenvalue of $\R(T)$ is at least $(|L_{h-j+1}|-|L_{h-j}|)$ for $j \in \psi$ and 1 for $j=h+1$.
\end{enumerate}
\end{Theorem}

We first prove some important Lemmas which are useful to prove our main results.  
 
\begin{Lemma}\label{lm:1} Let $T^z_{m_0,m_1,\ldots,m_{h-1}}\;(z=1,2)$ be the level-wise regular tree and $B_i$'s are as defined earlier. Let 
$$P = \begin{bmatrix}
\alpha_{1}\I_{|L_{h}|} & B_1 & 0 & \cdots & \cdots & \cdots & 0\\
B_1^{T} & \alpha_{2}\I_{|L_{h-1}|} & B_2 & \ddots & & \vdots & &\\
0 & B_2^{T} &  &  & \ddots &  &  &\\
\vdots & \ddots &  &  &  & \ddots &  &  &\\
\vdots &  & \ddots &  & \alpha_{h-1}\I_{|L_2|} & B_{h-1} & 0\\
\vdots &  &  & \ddots & B_{h-1}^{T} & \alpha_{h}\I_{|L_1|} & B_{h}\\
0 & \cdots & \cdots & \cdots & 0 &  B_{h}^{T} & \Omega'
\end{bmatrix}, 
$$
where 
$$
\Omega' = \left\{
\begin{array}{ll}
\alpha_{h+1}, & \mbox{ if $z=1$} , \\ [0.2cm]
\begin{bmatrix}
\alpha_{h+1} & \frac{1}{m_0} \\
\frac{1}{m_0} & \alpha_{h+1}
\end{bmatrix}, & \mbox{ if $z=2$}.
\end{array}
\right.
$$
Let $\beta_1=\alpha_1$
and 
$\beta_j=\alpha_j-\frac{|L_{h-j+2}|}{|L_{h-j+1}|}\frac{1}{(m_{h-j+1}m_{h-j+2})\beta_{j-1}}, j=2,3,\dots,h+1$.
If $\beta_{j} \neq 0$ for all $j=1,2,\dots,h$ and 
$$
\beta_{h+1} \neq \left\{
\begin{array}{ll}
0, & \mbox{ if $z=1$} , \\ [0.2cm]
\pm\frac{1}{m_0} , & \mbox{ if $z=2$}.
\end{array}
\right.
$$
Then
\begin{equation}\label{eq:2}
\det P = \left\{
\begin{array}{ll}
\beta_1^{|L_h|}\beta_2^{|L_{h-1}|}\cdots\beta_{h-1}^{|L_2|}\beta_{h}^{|L_1|}\beta_{h+1}, & \mbox{ if $z=1$} , \\ [0.2cm]
\beta_1^{|L_h|}\beta_2^{|L_{h-1}|}\cdots\beta_{h-1}^{|L_2|}\beta_{h}^{|L_1|}\left(\beta_{h+1}+\frac{1}{m_0}\right)\left(\beta_{h+1}-\frac{1}{m_0}\right), & \mbox{ if $z=2$}.
\end{array}
\right.
\end{equation}
\begin{proof}
Suppose $\beta_j\neq 0$ for all $j=1,2,\dots,h+1$. We apply Gauss elimination procedure to reduce the matrix $P$ to an upper triangular matrix without row interchanges. Then by Observation \ref{Obs1}, we have the following matrix just before the last step.  
$$\begin{bmatrix}
\beta_{1}\I_{|L_h|} & B_1 & 0 & \cdots & \cdots & \cdots & 0\\
0 & \beta_{2}\I_{|L_{h-1}|} & B_2 & \ddots & & & \vdots \\
0 & 0 &  &  & \ddots &  &  \vdots\\
\vdots & \ddots &  &  &  & \ddots & \vdots\\
\vdots &  & \ddots &  & \beta_{h-1}\I_{|L_2|} & B_{h-1} & 0\\
\vdots &  &  & \ddots & 0 & \beta_{h}\I_{|L_1|} & B_{h}\\
0 & \cdots & \cdots & \cdots & 0 & B_{h}^{T} & \Omega'
\end{bmatrix}$$

We now consider the following two cases.

\textsf{Case-1:} $z=1$ (or $T^1_{m_0,m_1,\ldots,m_{h-1}}$).

In this cases, note that $\Omega' = \alpha_{h+1}$ and hence applying the final step of Gauss elimination gives 
$$\begin{bmatrix}
\beta_{1}\I_{|L_h|} & B_1 & 0 & \cdots & \cdots & \cdots & 0\\
0 & \beta_{2}\I_{|L_{h-1}|} & B_2 & \ddots & & & \vdots\\
0 & 0 &  &  & \ddots &  &  \vdots \\
\vdots & \ddots &  &  &  & \ddots &  \vdots \\
\vdots &  & \ddots &  & \beta_{h-1}\I_{|L_2|} & B_{h-1} & 0\\
\vdots &  &  & \ddots & 0 & \beta_{h}\I_{|L_1|} & B_{h}\\
0 & \cdots & \cdots & \cdots & 0 & 0 & \alpha_{h+1}-|L_1|\frac{1}{(m_1m_0)\beta_h}
\end{bmatrix} $$
Taking determinant of the above matrix, we obtain
\begin{equation*}
\det P = \beta_1^{|L_h|}\beta_2^{|L_{h-1}|}\ldots\beta_{h-1}^{|L_2|}\beta_h^{|L_1|}\beta_{h+1}.
\end{equation*}

\textsf{Case-2:} $z=2$ (or $T^2_{m_0,m_1,\ldots,m_{h-1}}$).

In this case, note that $\Omega' = \begin{bmatrix}
    \alpha_{h+1} & \frac{1}{m_0} \\
    \frac{1}{m_0} & \alpha_{h+1}
\end{bmatrix}$ and hence applying the final step of Gauss elimination gives
$$\begin{bmatrix}
\beta_{1}\I_{|L_h|} & B_1 & 0 & \cdots & \cdots & \cdots & 0\\
0 & \beta_{2}\I_{|L_{h-1}|} & B_2 & \ddots & & & \vdots\\
0 & 0 &  &  & \ddots &  &  \vdots \\
\vdots & \ddots &  &  &  & \ddots & \vdots\\
\vdots &  & \ddots &  & \beta_{h-1}\I_{|L_2|} & B_{h-1} & 0\\
\vdots &  &  & \ddots & 0 & \beta_{h}\I_{|L_1|} & B_{h}\\
0 & \cdots & \cdots & \cdots & 0 & 0 & \Omega'-\frac{1}{\beta_h} B^T_h B_h
\end{bmatrix}.$$
Taking determinant of the above matrix, we obtain
\begin{equation*}
\det P = \beta_1^{|L_h|}\beta_2^{|L_{h-1}|}\ldots\beta_{h-1}^{|L_2|}\beta_h^{|L_1|}\det\left(\Omega'-\frac{1}{\beta_h}B^T_h B_h\right).
\end{equation*}
Observe that $\Omega' - \frac{1}{\beta_h} B^T_h B_h = \begin{bmatrix}
   \alpha_{h+1} & \frac{1}{m_0} \\
   \frac{1}{m_0} & \alpha_{h+1}
\end{bmatrix} - \frac{1}{\beta_h} \begin{bmatrix}
   \frac{|L_1|}{|L_0|}\frac{1}{m_1m_0} & 0 \\
   0 & \frac{|L_1|}{|L_0|}\frac{1}{m_1m_0}
\end{bmatrix} = \begin{bmatrix}
    \beta_{h+1} & \frac{1}{m_0} \\
    \frac{1}{m_0} & \beta_{h+1}
\end{bmatrix}$. Hence, $\det(\Omega'-\frac{1}{\beta_h}B^T_h B_h) = \beta^2_{h+1}-\frac{1}{m^2_0} = \(\beta_{h+1}-\frac{1}{m_0}\)\(\beta_{h+1}+\frac{1}{m_0}\)$. Substituting this in above, we get the result.
\end{proof}
\end{Lemma}

Denote $[1,h] = \{1,2,3,\ldots,h\}$. Define a subset $\psi$ of $[1,h]$ as $\psi= \{j \in [1,h] : |L_{h-j+1}| > |L_{h-j}|\}$. Since $|L_{1}|>|L_{0}|=1$, the index $h \in \psi$ and hence $\psi \neq \emptyset$. Observe that if $l \in [1,h] \setminus \psi$ then $|L_{h-l+1}|=|L_{h-l}|$ and in this case, $C_l$ is of size $1 \times 1$ and by the definition of $B_j$, we have 
$$B_l = 
\begin{bmatrix}
C_l & 0 & \cdots & 0\\
0 & C_l & &\vdots\\
\vdots & & \ddots & 0\\
0 & \cdots & 0 & C_l
\end{bmatrix},
$$
where $B_l$ is of size $|L_{h-l}| \times |L_{h-l}|$. That is,  $B_l=\frac{1}{\sqrt{(m_{h-l} m_{h-l+1})}}\I_{|L_{h-l}|}$.

\begin{Lemma}\label{Th:1} Let $T^z = T^z_{m_0,m_1,\ldots,m_{h-1}}\;(z=1,2)$ and for $i=0,1,\ldots,h$, $L_i$'s are defined as earlier. 
Let $$\p_0(\lambda)=1, \p_1(\lambda)=\lambda$$ 
and for $j=2,3,\ldots,h+1$, 
\begin{equation}\label{eq:4}
\p_j(\lambda)=\lambda \p_{j-1}(\lambda)-\frac{|L_{h-j+2}|}{|L_{h-j+1}|(m_{h-j+2}m_{h-j+1})}\p_{j-2}(\lambda).   
\end{equation}
Then 
\begin{enumerate}[\rm (a)]
\item If $\p_j(\lambda) \neq 0$ for all $j=1,2,\dots,h$ and $\p_{h+1}(\l) \neq 0$ if $z=1$ and $\p_{h+1}(\l) \neq \pm\frac{1}{m_0}$ if $z=2$. 
Then $\det(\lambda \I-\R(T^z))$
\begin{equation}\label{eq:3}
 = \left\{
\begin{array}{ll}
\p_{h+1}(\lambda) \displaystyle\prod_{j\in \psi} \p_j^{|L_{h-j+1}|-|L_{h-j}|}(\lambda), & \mbox{if $z=1$} , \\ [0.2cm]
\(\p_{h+1}(\lambda)-\frac{1}{m_0}\p_h(\l)\)\(\p_{h+1}(\lambda)+\frac{1}{m_0}\p_h(\l)\) \displaystyle\prod_{j\in \psi} \p_j^{|L_{h-j+1}|-|L_{h-j}|}(\lambda), & \mbox{if $z=2$}.
\end{array}
\right.
\end{equation}
\item $\sigma(R(T^z))=
\left\{
\begin{array}{ll}
\left(\displaystyle\bigcup_{j\in\psi}\{\lambda\in \mathbb{R}: \p_j(\lambda)=0\}\right)\bigcup\{\lambda\in\mathbb{R}:\p_{h+1} (\lambda)=0\}, & \mbox{if $z=1$}, \\[0.8cm]
\left(\displaystyle\bigcup_{j\in\psi}\{\lambda\in \mathbb{R}: \p_j(\lambda)=0\}\right)\bigcup \{\l \in \mathbb{R} : \p_{h+1}(\lambda)\pm\frac{1}{m_0}\p_h(\l)=0\}, & \mbox{if $z=2$}.
 \end{array}
\right.
$
\end{enumerate}
\end{Lemma}
\begin{proof} (a) We apply Lemma \ref{lm:1} to the matrix $P=\lambda\I-\R(T^z)$. For this matrix $\alpha_j = \lambda$ for $j=1,2,3,\ldots,h+1$. Let $\beta_1,\beta_2,\ldots,\beta_{h+1}$ be as in the Lemma \ref{lm:1}. Suppose that $\lambda \in \mathbb{R}$ is such that $\p_j(\lambda) \neq 0$ for all $j=1,2,\dots,h+1$. We have
\bean
\b_1 & = & \l = \frac{\p_1(\l)}{\p_0(\l)} \\
\b_2 & = & \l-\frac{|L_h|}{|L_{h-1}|(m_{h-1}m_h)}\frac{1}{\b_1} = \l-\frac{|L_h|}{|L_{h-1}|(m_{h-1}m_h)}\frac{\p_0(\l)}{\p_1(\l)} \\ & = & \frac{\l \p_1(\l)-\frac{|L_h|}{|L_{h-1}|(m_{h-1}m_h)}\p_0(\l)}{\p_1(\l)} = \frac{\p_2(\l)}{\p_1(\l)} \neq 0 \\
\b_3 & = & \l-\frac{|L_{h-1}|}{|L_{h-2}|(m_{h-2}m_{h-1})}\frac{1}{\b_2} = \l-\frac{|L_{h-1}|}{|L_{h-2}|(m_{h-2}m_{h-1})}\frac{\p_1(\l)}{\p_2(\l)} \\ 
& = & \frac{\l \p_2(\l)-\frac{|L_{h-1}|}{|L_{h-2}|(m_{h-2}m_{h-1})}\p_1(\l)}{\p_2(\l)} = \frac{\p_3(\l)}{\p_2(\l)} \neq 0 \\
\vdots & & \vdots \\
\b_{h} & = & \l-\frac{|L_2|}{|L_1|(m_2m_1)}\frac{1}{\b_{h-1}} = \l-\frac{|L_2|}{|L_1|(m_2m_1)} \frac{\p_{h-2}(\l)}{\p_{h-1}(\l)} \\ 
& = & \frac{\l \p_{h-1}(\l)-\frac{|L_2|}{|L_1|(m_2m_1)}\p_{h-2}(\l)}{\p_{h-1}(\l)} = \frac{\p_{h}(\l)}{\p_{h-1}(\l)} \neq 0 \\
\b_{h+1} & = & \l-\frac{|L_1|}{|L_0|(m_1m_0)}\frac{1}{\b_{h}} = \l-\frac{|L_1|}{|L_0|(m_1m_0)}\frac{\p_{h-1}(\l)}{\p_{h}(\l)} \\
& = & \frac{\l \p_{h}(\l)-\frac{|L_1|}{|L_0|(m_1m_0)}\p_{h-1}(\l)}{\p_{h}(\l)} = \frac{\p_{h+1}(\l)}{\p_{h}(\l)}.
\eean

We consider the following two cases. 

\textsf{Case-1:} $z=1$ (or $T^1_{m_0,m_1,\ldots,m_{h-1}}$).

By \eqref{eq:2}, we have
\bean    
\det(\l \I-\R(T^1)) & = &  \beta_1^{|L_h|}\beta_2^{|L_{h-1}|}\cdots \beta_{h-1}^{|L_2|}\beta_{h}^{|L_1|}\beta_{h+1} \\
& = & \frac{\p_1^{|L_h|}(\l)}{\p_0^{|L_h|}(\l)} \frac{\p_2^{|L_{h-1}|}(\l)}{\p_1^{|L_{h-1}|}(\l)} \frac{\p_3^{|L_{h-2}|}(\l)}{\p_2^{|L_{h-2}|}(\l)}\cdots \frac{\p_h^{|L_1|}(\l)}{\p_{h-1}^{|L_1|}(\l)} \frac{\p_{h+1}(\l)}{\p_{h}(\l)} \\
& = & \p_1^{|L_h|-|L_{h-1}|}(\l)\p_2^{|L_{h-1}|-|L_{h-2}|}(\l)\p_3^{|L_{h-2}|-|L_{h-3}|}(\l)\cdots \p_{h}^{|L_1|-1}(\l)\p_{h+1}(\l) \\
& = & \p_{h+1}(\l) \prod_{j\in \psi} \p_j^{|L_{h-j+1}|-|L_{h-j}|}(\l).
\eean 

\textsf{Case-2:} $z=2$ (or $T^2_{m_0,m_1,\ldots,m_{h-1}}$).

By \eqref{eq:2}, we have
\bean
\det(\l \I -\R(T^2)) & = & \beta_1^{|L_h|}\beta_2^{|L_{h-1}|}\cdots \beta_{h-1}^{|L_2|}\beta_{h}^{|L_1|}\(\beta_{h+1}-\frac{1}{m_0}\)\(\beta_{h+1}+\frac{1}{m_0}\) \\
& = & \frac{\p_1^{|L_h|}(\l)}{\p_0^{|L_h|}(\l)} \frac{\p_2^{|L_{h-1}|}(\l)}{\p_1^{|L_{h-1}|}(\l)} \frac{\p_3^{|L_{h-2}|}(\l)}{\p_2^{|L_{h-2}|}(\l)}\cdots \frac{\p_h^{|L_1|}(\l)}{\p_{h-1}^{|L_1|}(\l)} \(\frac{\p_{h+1}(\l)}{\p_{h}(\l)}-\frac{1}{m_0}\)\(\frac{\p_{h+1}(\l)}{\p_{h}(\l)}+\frac{1}{m_0}\) \\
& = & \p_1^{|L_h|-|L_{h-1}|}(\l)\p_2^{|L_{h-1}|-|L_{h-2}|}(\l)\p_3^{|L_{h-2}|-|L_{h-3}|}(\l)\cdots \p_{h}^{|L_1|-2}(\l) \cdot \\
& &\(\p_{h+1}(\l)-\frac{1}{m_0}\p_h(\l)\)\(\p_{h+1}(\l)+\frac{1}{m_0}\p_h(\l)\) \\
& = & \(\p_{h+1}(\lambda)-\frac{1}{m_0}\p_h(\l)\)\(\p_{h+1}(\lambda)+\frac{1}{m_0}\p_h(\l)\)\displaystyle\prod_{j\in \psi} \p_j^{|L_{h-j+1}|-|L_{h-j}|}(\lambda).
\eean

(b) We consider the following two cases. 

\textsf{Case-1:} $z=1$ (or $T^1 = T^1_{m_0,m_1,\ldots,m_{h-1}}$).

From (\ref{eq:3}), if $\l\in\mathbb{R}$ is such that $\p_{j}(\l)\neq0$ for all $j=1,2,\dots,h+1$, then $\det(\l \I-\R(T^1)) \neq 0$. That is, 
$$\bigcap_{j=1}^{h+1}\left\{\l\in\mathbb{R}: \p_j(\l)\neq0\right\}\subseteq (\sigma(\R(T^1)))^{c}.$$ 
That is, 
$$\sigma(\R(T^1))\subseteq \left (\bigcup_{j=1}^{h}\left\{\l\in\mathbb{R}:\p_j(\l)=0\right\}\right)\cup\left\{\l \in \mathbb{R}: \p_{h+1}(\l)=0\right\}.$$ 
We claim that 
$$\sigma(\R(T^1))\subseteq \left (\bigcup_{j\in \psi}^{}\left\{\l\in\mathbb{R}:\p_j(\l)=0\right\}\right)\cup\left\{\l\in R: \p_{h+1}(\l)=0\right\}.$$ 
If $\psi=[1,h]=\{1,2,\dots,h\}$ then nothing to prove. Suppose that $\psi$ is proper subset of $[1,h]$ then above equation is equivalent to $$\(\bigcap_{j\in \psi} \left\{\l\in\mathbb{R}: \p_{j}(\l)\neq 0\right\}\) \cap \left\{\l\in\mathbb{R}:\p_{h+1}(\l)\neq 0\right\} \subseteq \left(\sigma(\R(T^1))\right)^{c}.$$ 

Suppose that $\l \in \mathbb{R}$ is such that $\p_{j}(\l)\neq 0$ for all $j\in \psi$ and $\p_{h+1}(\l) \neq 0$. Since $h \in \psi$, $\p_{h}(\l)\neq 0$. If $\p_{j}(\l) \neq 0$ for all $j \in [1,h] \setminus \psi$ then $\det(\l\I-\R(T^1)) \neq 0$ that is $\l \in \left(\sigma(\R(T^1))\right)^{c}$. If $\p_{i}(\l)=0$ for some $i \in [1,h] \setminus \psi$, let $l$ be the first index in $[1,h] \setminus \psi$ such that $\p_{l}(\l)=0$. Then $\b_{j}\neq 0$ for all $j=1,2,\dots,l-1$ and $\b_{l}=0$. Now taking $j=l+2$ in \eqref{eq:4}, we obtain
\begin{equation}\label{eq:4b}
\p_{l+2}(\l)=\l \p_{l+1}(\l).    
\end{equation}
Note that $\p_{l+1}(\l) \neq 0$; otherwise taking $j=l+1$ in \eqref{eq:4}, we have 
\begin{equation*}
\p_{l+1}(\l) = \l\phi_{l}(\l)-\frac{|L_{h-l+1}|}{|L_{h-l}|(m_{h-l+1}m_{h-l})}\phi_{l-1}(\l)
\end{equation*}
which gives $\phi_{l-1}(\l)=0$ as $\phi_{l+1}(\l) = \phi_{l}(\l) = 0$ and $|L_{h-l+1}|, |L_{h-l}|, m_{h-l+1}, m_{h-l} \neq 0$, and continuing this back substitution in this way in \eqref{eq:4} gives $\p_{0}(\l)=0$, a contradiction. 

Hence, $\phi_{l+1}(\l) \neq 0$. Therefore, we have 
\bean
\b_{l+2} & = & \l-\frac{|L_{h-l}|}{|L_{h-l-1}|(m_{h-l}m_{h-l-1})}\frac{\p_{l}(\l)}{\p_{l+1}(\l)} \\
& = & \frac{\l\phi_{l+1}(\l)-\frac{|L_{h-l}|}{|L_{h-l-1}|(m_{h-l}m_{h-l-1})}\phi_l(\l)}{\phi_{l+1}(\l)} \\
& = & \frac{\phi_{l+2}(\l)}{\phi_{l+1}(\l)} \\
& = & \l \hspace{2cm}\mbox{(by \eqref{eq:4b}).}
\eean
Since $l \in [1,h]-\psi$, then $|L_{h-l+1}|=|L_{h-l}|$. That means $B_l=\frac{1}{\sqrt{(m_{h-l} m_{h-l+1})}}\I_{|L_{h-l+1}|}$ and the Gaussian elimination procedure applied to $P=\l\I+\R(T^1)$ yields to the intermediate matrix 
$$
 \begin{bmatrix}
\b_{1}\I_{|L_h|}&B_1&0&\cdots&\cdots&\cdots&0\\
0&\ddots&\ddots&\ddots& & &\vdots\\
\vdots&0&0&\frac{1}{\sqrt{m_{h-l} m_{h-l+1}}}\I_{|L_{h-l+1}|}&0& & \vdots\\
\vdots&\ddots&\frac{1}{\sqrt{m_{h-l} m_{h-l+1}}}\I_{|L_{h-l+1}|}&\l \I_{|L_{h-l}|}&B_{l+1}&\ddots&\vdots\\
\vdots& &\ddots&B_{l+1}^{T}&\l \I_{|L_{h-l-1}|}&\ddots&0\\
\vdots& & &\ddots&\ddots&\ddots&B_{h}\\
0&\cdots&\cdots&\cdots&0&B_{h}^{T}&\l
\end{bmatrix}
$$
$$ = 
\begin{bmatrix}
\b_{1}\I_{|L_h|}&B_1&0&\cdots&\cdots&\cdots&0\\
0&\ddots&\ddots&\ddots& & &\vdots\\
\vdots&0&0&\frac{1}{\sqrt{m_{h-l} m_{h-l+1}}}\I_{|L_{h-l+1}|}&0& & \vdots\\
\vdots&\ddots&\frac{1}{\sqrt{m_{h-l} m_{h-l+1}}}\I_{|L_{h-l+1}|}&\b_{l+1} \I_{|L_{h-l}|}&B_{l+1}&\ddots&\vdots\\
\vdots& &\ddots&B_{l+1}^{T}&\b_{l+2} \I_{|L_{h-l-1}|}&\ddots&0\\
\vdots& & &\ddots&\ddots&\ddots&B_{h}\\
0&\cdots&\cdots&\cdots&0&B_{h}^{T}&\b_{h+1}
\end{bmatrix}   
$$
Now, a number of $|L_{h-l+1}|$ rows interchanges gives the matrix
$$ = 
\begin{bmatrix}
\b_{1}\I_{|L_h|}&B_1&0&\cdots&\cdots&\cdots&0\\
0&\ddots&\ddots&\ddots& & &\vdots\\
\vdots&\ddots&\frac{1}{\sqrt{m_{h-l} m_{h-l+1}}}\I_{|L_{h-l+1}|}&\b_{l+1} \I_{|L_{h-l}|}&B_{l+1}&\ddots&\vdots\\
\vdots&0&0&\frac{1}{\sqrt{m_{h-l} m_{h-l+1}}}\I_{|L_{h-l+1}|}&0& & \vdots\\
\vdots& &\ddots&B_{l+1}^{T}&\b_{l+2} \I_{|L_{h-l-1}|}&\ddots&0\\
\vdots& & &\ddots&\ddots&\ddots&B_{h}\\
0&\cdots&\cdots&\cdots&0&B_{h}^{T}&\b_{h+1}
\end{bmatrix}$$
Therefore,
$$\det(\l\I+\R(T^1))= \(\frac{-1}{m_{h-l}m_{h-l+1}}\)^{|L_{h-l+1}|}\b_1^{|L_h|}\b_2^{|L_{h-1}|}\cdots\b_{l-1}^{|L_{h-l+2}|}\det
\begin{bmatrix}
\b_{l+2}\I_{|L_{h-l-1}|}&\ddots&0\\
\ddots&\ddots&B_{h}\\
0&B_{h}^{T}&\b_{l+2}
\end{bmatrix}
$$
Now, if there exists $j \in [1,h] \setminus \psi$, $l+2\leq j\leq h-1$, such that $\p_j(\l)=0$, we apply the above procedure to the matrix 
$$\begin{bmatrix}
\b_{l+2}\I_{|L_{h-l-1}|}&\ddots&0\\
\ddots&\ddots&B_{h}\\
0&B_{h}^{T}&\b_{l+2}
\end{bmatrix}
$$
Finally, we obtain ,
\begin{equation}\label{eq:6}
\det(\l \I+\R(T^1))=\gamma \cdot \b_{h+1} = \gamma \cdot \frac{\phi_{h+1}(\l)}{\phi_h(\l)}
\end{equation}
where $\gamma$ is different from $0$ and by hypothesis, $\phi_{h+1}(\l) \neq 0$ and $\phi_h(\l) \neq 0$. Therefore, $$\det(\l\I+\R(T^1)) \neq 0.$$ 
That means $\l \in \sigma((\R(T^1))^{c}$.

Now, we claim that 
$$\left (\bigcup_{j\in \psi}^{}\left\{\l\in\mathbb{R}:\p_j(\l)=0\right\}\right)\cup\left\{\l\in R: \p_{h+1}(\l)=0\right\}\subseteq\sigma(\R(T^1)).$$
Let $\l\in \left (\displaystyle\bigcup_{j\in \psi}^{}\left\{\l\in\mathbb{R}:\p_j(\l)=0\right\}\right)\cup\left\{\l\in \mathbb{R}: \p_{h+1}(\l)=0\right\}.$ Let $l$ be the first index in $\psi$ such that $\p_{l}(\l)=0.$ Then $\b_{l}= \frac{\p_{l}(\l)}{\p_{l-1}(\l)}=0.$ The corresponding intermediate matrix in the Gaussian elimination procedure applied to the matrix $P=\l\I+\R(T^1)$ is 
\begin{equation}\label{Mat:2}
\begin{bmatrix}
\b_{1}\I_{|L_h|}&B_1&0&\cdots&\cdots&0\\
0&\ddots&\ddots& & & \vdots\\
0&\ddots&0&B_{l}&0& \vdots\\
\vdots&\ddots&B_{l}^{T}&\l \I_{|L_{h-l}|}& &\vdots\\
\vdots& &\ddots&\ddots&\ddots&B_{h}\\
0&\cdots&\cdots&0&B_{h}^{T}&\l\\
\end{bmatrix}.
\end{equation}
Since $l \in \psi$, $|L_{h-l+1}| > |L_{h-l}|$ and $B_{l}$ is a matrix with more rows than columns. Therefore, the matrix in \eqref{Mat:2} has at least two equal rows. Thus $\det(\l\I+\R(T^1))=0.$ That is, $\l \in \sigma(\R(T^1)).$ Hence, we obtain  
\begin{equation}\label{eq:5}
\left (\bigcup_{j\in\psi}\left\{\l\in\mathbb{R}:\p_j(\l)=0\right\}\right)\subseteq\sigma(\R(T^1)).
\end{equation}
Now let $\l\in\left\{\l\in \mathbb{R}: \p_{h+1}(\l)=0\right\}.$ Observe that $\p_{h}(\l) \neq 0$; otherwise back substitution in \eqref{eq:4} yields to $\p_0(\l)=0$. If $\p_j(\l)=0$ for some $j \in \psi$ then the use of \eqref{eq:5} gives $\l \in \sigma(\R(T^1))$. Hence $\p_j(\l) \neq 0$ for all $j \in \psi$. If in addition $\p_j(\l) \neq 0$ for all $j \in \psi \setminus [1,h]$ then \eqref{eq:3} holds and thus $\det(\l\I+\R(T))=0$ because $\p_{h+1}(\l)=0$. Thus we obtain $\l \in \sigma(\R(T^1))$. If $\p_i(\l)=0$ for some $i \in \psi \setminus [1,h]$ then we have the assumptions under which \eqref{eq:6} was obtained gives 
$$\det(\l\I+\R(T^1)) = \gamma \cdot \b_{h+1} = \gamma \cdot \frac{\p_{h+1}(\l)}{\p_h(\l)} = 0.$$ 
Since $\det(\l \I-\R(T^1))=\det(\l\I +\R(T^1))$ by Lemma \ref{lem2}, we obtain $\l \in \sigma(\R(T^1))$. 

\textsf{Case-2:} $z=2$ (or $T^2 = T^2_{m_0,m_1,\ldots,m_{h-1}}$).

From \eqref{eq:3}, if $\l \in \mathbb{R}$ is such that $\p_j(\l) \neq 0$ for all $j=1,2,\ldots,h$ and $\p_{h+1}(\l) \pm \frac{1}{m_0}\p_h(\l) \neq 0$ then $\det(\l\I-\R(T^2)) \neq 0$. That is, 
\begin{equation*}
\(\displaystyle\bigcap_{j=1}^h\{\l \in \mathbb{R} : \p_j(\l) \neq 0\}\) \cap \{\l \in \mathbb{R} : \p_{h+1}(\l)\pm\frac{1}{m_0}\p_h(\l) \neq 0\} \subseteq (\sigma(\R(T^2)))^c.    
\end{equation*}
That is, 
\begin{equation*}
\sigma(\R(T^2)) \subseteq \(\bigcup_{j=1}^h\{\l \in \mathbb{R} : \p_j(\l) = 0\}\) \cup \{\l \in \mathbb{R} : \p_{h+1}(\l)\pm\frac{1}{m_0}\p_h(\l) = 0\}.
\end{equation*}
We claim that
\begin{equation*}
\sigma(\R(T^2)) \subseteq \(\bigcup_{j \in \psi}\{\l \in \mathbb{R} : \p_j(\l) = 0\}\) \cup \{\l \in \mathbb{R} : \p_{h+1}(\l)\pm\frac{1}{m_0}\p_h(\l) = 0\}.
\end{equation*}

Obviously, if $\psi = [1,h] = \{1,2,\ldots,h\}$ then the claim is straight forward. Suppose that $\psi$ is proper subset of $[1,h]$ then above equation is equivalent to 
\begin{equation*}
\(\bigcap_{j \in \psi}\{\l \in \mathbb{R} : \p_j(\l) \neq 0\}\) \cap \{\l \in \mathbb{R} : \p_{h+1}(\l) \pm \frac{1}{m_0}\p_h(\l) \neq 0\} \subseteq (\sigma(\R(T^2)))^c.     
\end{equation*}

Suppose that $\l \in \mathbb{R}$ is such that $\p_j(\l) \neq 0$ for all $j \in \psi$ and $\p_{h+1}(\l) \pm \frac{1}{m_0}\p_h(\l) \neq 0$. Since $h \in \psi$, $\p_h(\l) \neq 0$. If $\p_i(\l) = 0$ for some $i \in [1,h] \setminus \psi$ then let $l$ be the first index in $[1,h] \setminus \psi$ such that $\p_l(\l)=0$. Then $\b_j \neq 0$ for all $j=1,2,\ldots,l-1$ and $\b_l = 0$. Now proceeding as in case-1, we obtain
\bea
\det(\l\I+\R(T^2)) & = & \gamma \cdot \det\begin{bmatrix}
\beta_{h+1} & \frac{1}{m_0} \\
\frac{1}{m_0} & \beta_{h+1}
\end{bmatrix} \nonumber\\
& = & \gamma \cdot \(\b_{h+1}-\frac{1}{m_0}\)\(\b_{h+1}+\frac{1}{m_0}\) \nonumber\\
& = & \gamma \cdot \(\frac{\p_{h+1}(\l)}{\p_h(\l)}-\frac{1}{m_0}\)\(\frac{\p_{h+1}(\l)}{\p_h(\l)}+\frac{1}{m_0}\) \nonumber\\
& = & \gamma \cdot \frac{\(\p_{h+1}(\l)-\frac{1}{m_0}\p_h(\l)\)\(\p_{h+1}(\l)+\frac{1}{m_0}\p_h(\l)\)}{\p_h^2(\l)} \label{eq:det}
\eea
where $\gamma$ is different from $0$ and by hypothesis, $\p_{h+1}(\l) \pm \frac{1}{m_0}\p_h(\l) \neq 0$ and $\p_h(\l) \neq 0$. Therefore, $\det(\l\I+\R(T^2)) \neq 0$. That is, $\l \in (\sigma(\R(T^2)))$. 

Now we claim that 
\begin{equation*}
\(\bigcup_{j \in \psi}\{\l \in \mathbb{R} : \p_{j}(\l) = 0\}\) \cup \{\l \in \mathbb{R} : \p_{h+1}(\l) \pm \frac{1}{m_0}\p_h(\l) = 0\} \subseteq \sigma(\R(T^2)).    
\end{equation*}

Let $\l \in \(\displaystyle\bigcup_{j \in \psi}\{\l \in \mathbb{R} : \p_j(\l) = 0\}\)$. Let $l$ be the first index in $\psi$ such that $\p_l(\l) = 0$. Then $\b_l = \frac{\p_h(\l)}{\p_{h+1}(\l)} = 0$. Again proceeding as in case-1, we obtain 
\begin{equation}\label{eq:8}
\(\bigcup_{j \in \psi}\{\l \in \mathbb{R} : \p_j(\l) = 0\}\) \subseteq \sigma(\R(T^2)).    
\end{equation}

Now let $\l \in \{\l \in \mathbb{R} : \p_{h+1}(\l) \pm \frac{1}{m_0}\p_h(\l) = 0\}$. Observe that $\p_h(\l) \neq 0$; otherwise back substitution in \eqref{eq:4} yields to $\p_0(\l) = 0$, a contradiction. If $\p_j(\l) = 0$ for some $j \in \psi$ then the use of \eqref{eq:8} gives $\l \in \sigma(\R(T^2))$. Hence, assume $\p_j(\l) \neq 0$ for all $j \in \psi$. If in addition $\p_j(\l) \neq 0$ for all $j \in \psi \setminus [1,h]$ then \eqref{eq:3} holds and thus $\det(\l\I-\R(T^2)) = 0$ because $\p_{h+1}(\l) \pm \frac{1}{m_0}\p_h(\l) = 0$. Thus we obtain $\l \in \sigma(\R(T^2))$. If $\p_i(\l) = 0$ for some $i \in \psi \setminus [1,h]$ then we have the assumption under which \eqref{eq:det} was obtain gives 
\bean 
\det(\l\I+\R(T^2)) & = & \gamma \cdot \frac{\(\p_{h+1}(\l)-\frac{1}{m_0}\p_h(\l)\)\(\p_{h+1}(\l)+\frac{1}{m_0}\p_h(\l)\)}{\p_h^2(\l)} = 0.
\eean
Since $\det(\l\I-\R(T^2)) = \det(\l\I+\R(T^2))$ by Lemma \ref{lem2}, we have $\l \in \sigma(\R(T^2))$. The proof is complete.
\end{proof}

\begin{Lemma}\label{lm:2}
For $j=1,2,3,\dots,h$, let $P_j$ be the $j \times j$ principal sub matrix of the $(h+1) \times (h+1)$ symmetric tridiagonal matrix 
$$ P_{h+1} = \begin{bmatrix}
0 & \sqrt{\left(\frac{m_{h-1}-1}{m_hm_{h-1}}\right)} & 0 & \cdots & \cdots & 0\\
\sqrt{\left(\frac{m_{h-1}-1}{m_hm_{h-1}}\right)} & 0 & \sqrt{\left(\frac{m_{h-2}-1}{m_{h-1}m_{h-2}}\right)} & \ddots & & \vdots\\
0 & \sqrt{\left(\frac{m_{h-2}-1}{m_{h-1}m_{h-2}}\right)} & \ddots & \ddots & \ddots & \vdots \\
\vdots & \ddots & \ddots & \ddots & \sqrt{\left(\frac{m_1-1}{m_2m_1}\right)} & 0\\
\vdots & & \ddots & \sqrt{\left(\frac{m_1-1}{m_2m_1}\right)} & 0 & \sqrt{\frac{1}{m_1}}\\
0 & \cdots & \cdots & 0 & \sqrt{\frac{1}{m_{1}}} & 0
\end{bmatrix}.$$
and for $z=1,2$,
$$P_{h+1}^z = \begin{bmatrix}
0 & \sqrt{\left(\frac{m_{h-1}-1}{m_hm_{h-1}}\right)} & 0 & \cdots & \cdots & 0\\
\sqrt{\left(\frac{m_{h-1}-1}{m_hm_{h-1}}\right)} & 0 & \sqrt{\left(\frac{m_{h-2}-1}{m_{h-1}m_{h-2}}\right)} & \ddots & & \vdots\\
0 & \sqrt{\left(\frac{m_{h-2}-1}{m_{h-1}m_{h-2}}\right)} & \ddots & \ddots & \ddots & \vdots \\
\vdots & \ddots & \ddots & \ddots & \sqrt{\left(\frac{m_1-1}{m_2m_1}\right)} & 0\\
\vdots & & \ddots & \sqrt{\left(\frac{m_1-1}{m_2m_1}\right)} & 0 & \sqrt{\left(\frac{m_0-1)}{m_1m_0}\right)}\\
0 & \cdots & \cdots & 0 & \sqrt{\left(\frac{m_0-1}{m_{1}m_0}\right)} & \frac{(-1)^z}{m_0}
\end{bmatrix},$$
and $\p_j(\l), j=1,2,\ldots,h+1$ are as defined in Lemma \ref{Th:1}. Then for $j=1,2,\ldots,h+1$, 
$$\det(\lambda \I-P_{j})=\p_j(\lambda).$$
and for $z=1,2$,
$$\det(\lambda \I-P^z_{h+1}) = \p_{h+1}(\l)+\frac{(-1)^z}{m_0}\p_h(\l).$$
\begin{proof} In Lemma \ref{Prop1}, taking $p_j=0$ for $j=1,2,3,\dots,h$ and $p_{h+1} = 0$ for $P_{h+1}$ and $p_{h+1} = \frac{(-1)^z}{m_0}$ for $P^z_{h+1}$ (where $z=1,2$), $q_j= \sqrt{\frac{|L_{h-j+1}|}{(m_{h-j+1}m_{h-j})|L_{h-j}|}}$ and in addition, using $\sqrt{\frac{|L_{h-j+1}|}{(m_{h-j+1}m_{h-j})|L_{h-j}|}} = \sqrt{\frac{m_{h-j}-1}{m_{h-j+1}m_{h-j}}}$ for $j=1,2,\dots,h$ and if $T^1 = T^1_{m_0,m_1,\ldots,m_{h-1}}$ then observe that $\sqrt{\frac{|L_1|}{(m_{0}m_1)|L_0|}}=\sqrt{\frac{|L_1|}{m_{0}m_{1}}}=\sqrt{\frac{m_0}{m_0m_1}}=\sqrt{\frac{1}{m_1}}$ in \eqref{eq:4} gives the polynomials $\p_j(\l),\; j=1,\dots,h$ which completes the proof. 
\end{proof}
\end{Lemma}

\begin{proof}[\textbf{Proof of Theorem 3.1}]
We recall that eigenvalues of any symmetric tridiagonal matrix with nonzero co-diagonal entries are simple. Then \rm(a) and \rm(b) both are immediate consequences of Lemma \ref{Th:1} and Lemma \ref{lm:2}. 
\end{proof}

\begin{Theorem}\label{thm:larg1} Let $T^z = T^z_{m_0,m_1,\ldots,m_{h-1}}\;(z=1,2)$ be the level-wise regular tree of order $n$ and $\r_1(T^z) \geq \r_2(T^z) \geq \ldots \geq \r_n(T^z)$ be the eigenvalues of $\R(T^z)\;(z=1,2)$. Then 
\begin{enumerate}[\rm (a)]
\item $\r_1(T^1) = 1 \in \sigma(P_{h+1})$ and $\r_1(T^2) = 1 \in \sigma(P^2_{h+1}),$   
\item The multiplicity of $0$ as an eigenvalue of $\R(T^z)$ is $\alpha + \displaystyle\sum_{j \in \psi \atop j = odd}\(|L_{h-j+1}|-|L_{h-j}|\)$, where $\alpha=1$ if $h$ is even and $0$ otherwise. 
\end{enumerate}
\end{Theorem}
\begin{proof} (a) The proof follows by using interlacing property of the eigenvalues of real symmetric matrix in Theorem \ref{thm:main} and observing that $\sigma(P^1_{h+1}) = -\sigma(P^2_{h+1})$.   

(b) Again the proof is straight forward by using the fact that  $\R(T^z)\;(z=1,2)$ is a real symmetric matrix (hence the eigenvalues are real) and $T^z\;(z=1,2)$ is bipartite graph (hence $\r_i = \r_{n-i+1}\;i=1,2,\ldots,n$) in Theorem \ref{thm:main}.
\end{proof}

\begin{Example}\label{Ex:3} Let $T^1_{4,4,3}$ be the tree as in Fig. \ref{EX:1}. For $T^1_{4,4,3}$, $h=3, m_3=1, m_2=3, m_1=4, m_0=4, |L_3|=24, |L_2|=12, |L_1|=4, |L_0| = 1$. 
\end{Example}
Hence, we have $P_4 = \begin{bmatrix}
0 & \sqrt{\left(\frac{2}{3}\right)} & 0 & 0\\
\sqrt{\left(\frac{2}{3}\right)} & 0 & \sqrt{\left(\frac{1}{4}\right)} & 0\\
0 & \sqrt{\left(\frac{1}{4}\right)} & 0 & \sqrt{\left(\frac{1}{4}\right)}\\
0 & 0 & \sqrt{\left(\frac{1}{4}\right)} & 0 
\end{bmatrix}$ \\ 

Using Theorem \ref{thm:main}, the eigenvalues of Randić matrix of $T^1_{4,4,3}$ with its multiplicity are given in the following table.
$$
\begin{array}{|l|l|l|}
\hline
{\rm Matrix} & {\rm Eigenvalues}& {\rm Multiplicity} \\ \hline\hline
P_1 & 0 & (|L_3|-|L_2|) = 24-12 = 12 \\[0.3cm]
P_2 & -\frac{\sqrt{6}}{3}, \frac{\sqrt{6}}{3} & (|L_2|-|L_1|) = 12-4 = 8 \\[0.3cm]
P_3 & -\frac{\sqrt{33}}{6}, 0,  \frac{\sqrt{33}}{6} & (|L_1|-|L_0|) = 4-1 = 3 \\[0.3cm]
P_4 & -\frac{1}{\sqrt{6}},-1,1,\frac{1}{\sqrt{6}} & 1 \\[0.3cm]
\hline
\end{array}
$$

\begin{Example}\label{Ex:4} Let $T^2_{4,3,4}$ be the tree as in Fig. \ref{EX:2}. For $T^2_{4,3,4}$, $h=3$, $m_3=1, m_2 = 4, m_1 = 3, m_0 = 4$ and $|L_3| = 36, |L_2| = 12, |L_1| = 6, |L_0|=2$.  
\end{Example}
Hence, for $z = 1,2$, we have $P_{4}^{z} = \begin{bmatrix}
0 & \sqrt{\left(\frac{3}{4}\right)} & 0 & 0\\
\sqrt{\left(\frac{3}{4}\right)} & 0 & \sqrt{\left(\frac{1}{6}\right)} & 0\\
0 & \sqrt{\left(\frac{1}{6}\right)} & 0 & \sqrt{\left(\frac{1}{4}\right)}\\
0 & 0 & \sqrt{\left(\frac{1}{4}\right)} & \frac{(-1)^z}{4} 
\end{bmatrix}$  \\

Using Theorem \ref{thm:main}, the eigenvalues of Randić matrix of $T^2_{4,3,4}$ with its multiplicity are given in the following table.
$$
\begin{array}{|l|l|l|}
\hline
{\rm Matrix} & {\rm Eigenvalues} & {\rm Multiplicity} \\ \hline\hline
P_1 & 0 & \left(|L_3|-|L_2|\right) = 36-12 = 24 \\ [0.3cm]
P_2 & -\frac{\sqrt{3}}{2}, \frac{\sqrt{3}}{2} & \left(|L_2|-|L_1|\right) = 12-6 = 6 \\ [0.3cm]
P_3 & -\frac{\sqrt{33}}{6}, 0,  \frac{\sqrt{33}}{6} & (|L_1|-|L_0|) = 6-2 = 4 \\ [0.3cm]
P^1_4 & -1, -0.5672, 0.3373, 0.9799 & 1 \\ [0.3cm]
P^2_4 & -0.9799, -0.3373, 0.5672, 1 &  1 \\
\hline
\end{array}
$$

The star $K_{1,n-1}$ is a level-wise regular trees $T^1_{n-1,1}$. The Randić spectra of $K_{1,n-1}$ is given in \cite{Alikhani2015} as follows:
\begin{Theorem}\cite{Alikhani2015} Let $K_{1,n-1}$ be the star on $n$ vertices. Then
$S_{R}(K_{1,n-1}) = \bigg(\begin{matrix}
-1 & 0 & 1  \\
1 & n-2 & 1 
\end{matrix}\bigg)$.
\end{Theorem}

The above result can also be proved using our main result Theorem \ref{thm:main}. The outline of the proof is as follows: For a star graph $K_{1,n-1}$ on $n$ vertices, note that $h=1$.

Hence, we have $P_{2} = 
\begin{bmatrix}
  0 & 1 \\
  1 & 0
\end{bmatrix}.$

Using Theorem \ref{thm:main}, the Randić spectra of $K_{1,n-1}$ is given as \\
$$\begin{array}{|l|l|l|}
\hline
{\rm Matrix} & {\rm Eigenvalues} & {\rm Multiplicity} \\ \hline\hline
P_1 & 0 & (|L_1|-|L_0|) = (n-1)-1 = n-2 \\[0.3cm]
P_2 & -1, 1 & 1 \\ \hline
\end{array}
$$
\smallskip

\section*{Concluding Remarks}

The results presented in this paper reduce the computation of the eigenvalues of the Randić matrix of level-wise regular trees of size equal to the order (given in \eqref{eq:order}) of level-wise regular trees to specific tridiagonal matrices (described in Theorem \ref{thm:main}) of size equal to the height of level-wise regular trees plus one. For examples, the level-wise regular trees given in Examples \ref{EX:1} and \ref{EX:2} are of order 41 and 56, respectively (and hence their Randić matrices of size $41 \times 41$  and $56 \times 56$, respectively), but the eigenvalues for both trees can be calculated using matrices of size $4 \times 4$ (refer Examples  \ref{Ex:3} and \ref{Ex:4}). Hence, the Randić index (which is the sum of the squares of the eigenvalues of Randić matrix) and the Randić energy (which is the sum of the absolute values of the eigenvalues of Randić matrix) can be easily calculated for level-wise regular trees.

\end{document}